\documentclass[10pt]{amsart}
\usepackage{graphicx} 
\usepackage{amsfonts,amsmath,tikz,mathabx,caption}
\usepackage{amssymb}
\usepackage{pb-diagram}
\usepackage[all]{xy}
\usepackage{hyperref,xcolor}
\usetikzlibrary{decorations.markings}
\usetikzlibrary{shapes}
\usetikzlibrary{backgrounds}
\usepackage{subfig}

  \newtheorem{thm}{Theorem}
\newtheorem{lemma}[thm]{Lemma}

\newtheorem{cor}[thm]{Corollary}
\newtheorem{prop}[thm]{Proposition}

\numberwithin{thm}{section}

\newcommand{\Z}{\mathbb{Z}}
\newcommand{\N}{\mathbb{N}}

\newcommand{\bZ}{\mathbb{Z}}

\newcommand{\Sym}{\textup{Sym}}

\newcommand{\Supp}{\textup{Supp}}
  
\title[MIF and highly transitive overgroups]{Embeddings into highly transitive and mixed identity free groups}
\author{James Hyde\qquad Yash Lodha}
\date{\today}

\begin{document}

\begin{abstract}
Given a countable group $G$, we develop a method to construct an overgroup $H$ that is finitely generated, highly transitive and mixed identity free.
    Our construction can be controlled to ensure that some fundamental group theoretic properties of $G$ are inherited by $H$, such as amenability or the property of not containing a nonabelian free group. The former provides a strong solution to a question of Hull and Osin, and the latter provides the first examples of nonamenable groups without free subgroups that are highly transitive and mixed identity free.
    Our examples also have a nontrivial amenable radical, answering a question of Arzhantseva.
\end{abstract}

\maketitle

\section{Introduction.}

An action of a group $G$ on a set $\Omega$ is \emph{$k$-transitive} if $|\Omega|\geq k$ and for every pair of $k$-tuples $(x_1,...,x_k),(y_1,...,y_k)\in \Omega^k$ where $x_i\neq x_j, y_i\neq y_j$ for $i\neq j$, there is an element $g\in G$ such that $g\cdot x_i=y_i$ for each $1\leq i\leq k$. 
The transitivity degree of a countable group $G$, denoted $\textup{td}(G)$,
is the supremum of all $k\in \N$ such that $G$ admits a $k$-transitive faithful action.
A faithful action is \emph{highly transitive} if it is $k$-transitive for each $k\in \N$.
Infinite groups that admit highly transitive actions are called highly transitive groups.
Equivalently, a countably infinite group
is highly transitive if and only if it embeds as a dense subgroup in the infinite symmetric
group $\Sym(\N)$ endowed with the topology of pointwise convergence.
For the permutation groups $S_n,A_n$, clearly $td(S_n)=n$,
and $td(A_n)=n-2$. It is a consequence of the classification of finite simple groups that any
finite group $G$ other than $S_n$ or $A_n$ has $td(G)\leq 5$ (see \cite{MR1409812}).
In \cite{HullOsin}, Hull and Osin initiated a study of transitivity degree for infinite groups,
where even some basic questions remain open.
A particularly interesting direction of investigation is the question of whether there is a 
reasonable classification of highly transitive groups.
A fundamental observation of Hull and Osin in \cite{HullOsin} is the following dichotomy for a highly transitive faithful action of a group $G$ on a set $X$: either $G$ contains as a normal subgroup $\textup{Alt}_f(X)$, the group of finitely supported alternating permutations of $X$, or $G$ is mixed identity free.

A group $G$ has a \emph{non-trivial mixed identity} if there is a non-trivial word $W\in G\ast \bZ$ such that every homomorphism $G\ast \bZ\to G$ which restricts to the identity map on $G$ maps $W$ to the identity. 
If there is no such $W$, then $G$ is said to be \emph{mixed identity free}.
This is a fundamental property in group theory, and has recently witnessed some surprising applications in $C^*$-algebras (see \cite{amrutam2025strictcomparisonreducedgroup} and \cite{elayavalli2025negativeresolutioncalgebraictarski}), as well as striking connections with the Boone-Higman conjecture \cite{belk2025boonehigmanembeddingsmathrmautfnmapping}.

Given a finitely generated group $G$ that acts faithfully on $\N$, it is straightforward to construct a finitely generated overgroup $H$ that acts highly transitively on $\N$. However, the obvious construction leads to a group $H$ which satisfies a nontrivial mixed identity since it contains $\textup{Alt}_f(\N)$, the group of finitely supported permutations of $\N$. Other natural constructions of such overgroups do not preserve key properties of the group such as amenability or that of not containing a nonabelian free subgroup (for instance, taking a free product with $\Z$ \cite{HullOsin}). 
It is natural to inquire whether one can construct an overgroup $H$ that preserves such a property and is also mixed identity free (equivalently, does not contain a normal subgroup isomorphic to $\textup{Alt}_f(\N)$, see Theorem $5.9$ in \cite{HullOsin}). Note that $\N$ here may denote any countable set the group acts on, we use the natural numbers simply as a model.

The challenge associated with this question is illustrated by the following weaker question posed by Hull and Osin.
In \cite{HullOsin}, they inquire whether there is a finitely generated amenable group that is highly transitive and mixed identity free. 
It turns out that topological full groups of minimal cantor systems provide such examples (see \cite{bradford2023lengthnonsolutionsequationsconstants}). In this article, our first theorem provides a general construction:

\begin{thm}\label{main}
For every finitely generated group $G$, there is a finitely generated group $G\leq H\leq \Sym(G\oplus \Z)$ so that:
\begin{enumerate}
\item There is a short exact sequence: $$1\to N\to H\to (G\oplus \Z)\to 1$$ where $N$ is locally finite.
\item $H$ acts highly transitive on $G\oplus \Z$ via the left regular action.
\item $H$ is mixed identity free.
\end{enumerate}
\end{thm}

We obtain the following as a corollary.

\begin{cor}\label{cor:main1}
The following hold:
\begin{enumerate}
\item Every countable (elementary) amenable group embeds in a finitely generated (elementary) amenable group that is highly transitive and mixed identity free.
\item Every countable group without nonabelian free subgroups embeds in a finitely generated group without nonabelian free subgroups that is highly transitive and mixed identity free.
\end{enumerate}
\end{cor}

The von Neumann-Day problem asks whether there exist nonamenable groups 
without nonabelian free subgroups.
This was solved by Olshanskii (\cite{VNDAYFIRST}) and subsequently, many other counterexamples emerged (see \cite{OlshanskiiSapir},\cite{MonodPNAS}, \cite{LodhaMoore}).
It is natural to inquire whether there exist counterexamples that are highly transitive and mixed identity free.
Note that Tarski monsters and Burnside groups are not highly transitive (see \cite{HullOsin}), nor are the piecewise projective groups constructed in \cite{MonodPNAS}, \cite{LodhaMoore} or their variants (in fact they also satisfy a nontrivial mixed identity). 
As an application of our Theorem, we provide the first such examples.



\begin{cor}\label{cor:main2}
There exist (continuum many, up to isomorphism) finitely generated nonamenable groups without nonabelian free subgroups that are mixed identity free and highly transitive. Moreover, these groups also have a nontrivial amenable radical and hence are not $C^*$-simple.
\end{cor}

The above Corollary also answers a question of Arzhantseva (posed at the MF Oberwolfach workshop ``$C^\ast$-algebras’’, August 2025 \cite{Arzhantseva_personal}): Does there exist a nonamenable group $G$ such that $G$ is mixed identity free
and $C^*_r(G)$ is not selfless? Since selflessness of $C^*_r(G)$ implies $C^*$-simplicity of $G$ (see \cite{ozawa2025proximalityselflessnessgroupcalgebras}), the previous Corollary answers the question in the affirmative.

Our method of construction provides a new method of building amenable limits in the Grigorchuk space of marked groups, 
and this may find applications elsewhere.
This space has emerged to be quite useful for providing all kinds of new constructions and results (see \cite{MR4286046}, \cite{MR3644759}, \cite{MR4630775} for instance).

\section{The construction.}
Let $G_0$ be a finitely generated amenable group with a finite generating set $S_0$. We let: $$G=G_0\oplus (\bZ=\langle z\rangle)\qquad S=S_0\cup \{z\}$$
Let $\Gamma(G,S)$ be the respective Cayley graph with vertex set $G$ and edge set $\{(g,gs)\mid g\in G,s\in S\}$. This is endowed with the standard path metric, and for $g\in G$, we denote the usual word norm as $|g|_S$. $G$ acts by left multiplication on $\Gamma(G,S)$ by isometries. We denote by $B_n(\Gamma(G,S))$ the $n$-ball centered around the identity in the Cayley graph and by $B_n^v(\Gamma(G,S))$ we denote the set of elements in $G$ whose $S$-word norm is at most $n$ (i.e. the vertices in $B_n(\Gamma(G,S))$).

Recall that $G$ embeds in $\Sym(G)$ via the left regular action. $\Sym(G)$ is endowed with the topology of pointwise convergence. Given an element $f\in \Sym(G)$, the support of the element, or $\Supp(f)$ is the set of elements of $G$ that are moved by $f$.
We say that $f,g\in \Sym(G)$ are \emph{disjoint} if they have disjoint supports. 
Given an infinite collection of pairwise disjoint permutations $\{f_{m}\}_{m\in \N}$, we can consider their product $\prod_{m\in \N}f_{m}$ which is the limit $\lim_{n\to \infty}f_1...f_n$ in the topology of pointwise convergence. 
We denote by $\Sym_f(G)$ to be the set of finitely supported permutations.
An element $f\in \Sym(G)$ is called a \emph{involution} if it is an element of order two, or equivalently, in cycle notation it is a (possibly infinite) product of disjoint transpositions.

A key starting point of our construction is the following Lemma. 

\begin{lemma}\label{lemma:largecommutator}
Let $G,S,z$ be as above. Let $f_1,...,f_m\in \Sym_f(G)$ be involutions. 
Let $$S_1=S\cup \{f_1,...,f_m\}\qquad K_1=\langle S_1\rangle$$
For each $n\in \N$, we can choose a sufficiently large $\nu_n\in \N$ and define:
$$g_i=z^{\nu_n}f_iz^{-\nu_n} f_i\qquad S_2=S\cup \{g_1,...,g_m\}\qquad K_2=\langle S_2\rangle$$
such that the following holds:
\begin{enumerate}
\item $B_n(\Gamma(K_1,S_1))\cong B_n(\Gamma(K_2,S_2))$.
\item $g_1,...,g_m$ are involutions in $\Sym_f(G)$.
\item For $f\in B_n^v(\Gamma(K_1,S_1))$, we have that $f\cdot z^{\nu_n}=z^{\nu_n}(f\cdot id)$.
\end{enumerate}
\end{lemma}

\begin{proof}

Let $X_i$ be the (finite) support of $f_i$.
Since $z$ is an infinite order central element of $G$, for large $\nu_n$, the conjugate $z^{\nu_n}f_iz^{-\nu_n}$ is an involution whose support in $G$ (which equals $z^{\nu_n}\cdot \Supp(f_i)$) lies far from the identity in $\Gamma(G,S)$.
In particular, when: $$(z^{\nu_n}\cdot \Supp(f_i))\cap \Supp(f_i)=\emptyset$$ we have that $g_i=z^{\nu_n}f_iz^{-\nu_n} f_i$ is an involution.
The action of the permutation given by some word of the form $x_1y_1...x_ly_l$ for $2l\leq n$ where $$x_i\in S\cup\{id\}, y_i\in \{f_1^{\pm 1},...,f_m^{\pm 1}, id\}$$ on a neighborhood of the identity in $\Gamma(G,S)$ agrees with that of $$x_1(z^{\nu_n}y_1z^{-\nu_n})...x_l(z^{\nu_n}y_lz^{-\nu_n})$$ around the same size neighborhood centered at $z^{\nu_n}$ in $\Gamma(G,S)$, since $G$ acts on $\Gamma(G,S)$ homogeneously. 
This means that for a sufficiently large $\nu_n$, our claim holds.
\end{proof}

To construct our group, we will inductively define three sequences.
Two strictly increasing sequences of natural numbers $(\nu_n)_{n\in \N}, (k_n)_{n\in \N}$ and a sequence of involutions: $$(\sigma_s^{(n)})_{n\in \N}\subset \Sym_f(G)\qquad \text{ for each }s\in S$$ We denote: $$S_n=S\cup \{\sigma_s^{(n)}\mid s\in S\})\qquad H_n=\langle S_n\rangle$$
Moreover, for each $n\in \N$ we denote by $\Pi_n$ as the 
the normal closure of $\{\sigma_s^{(n)}\mid s\in S\}$ in $H_n$.
It will be demonstrated that for each $s\in S$, the limit $\lim_{n\to \infty} \sigma_n^{(n)}$ exists and will be denoted by $\sigma_s^{(\infty)}$.
We will introduce the generating set and the group: 
$$S_{\infty}=S\cup \{\sigma_s^{(\infty)}\mid s\in S\}\qquad H_{\infty}=\langle S_{\infty}\rangle$$ which will be our desired group.
Moreover, the group generated by: $$\{f^{-1}\sigma_s^{(\infty)} f\mid f\in H, s\in S\}$$ will be denoted as $\Pi_{\infty}$.

We introduce a proxy variable $\sigma_s$ for each $s\in S$,
and consider the set: $$P=S\cup \{\sigma_s\mid s\in S\}$$ and the group $F(P)$ that is freely generated by $P$.
For each reduced word $W$ in the generating set $P$ and $n\in \N\cup \{\infty\}$, we denote by $\xi_n(W)$ as the word in the generating set $S_n$ obtained from replacing each occurrence of $(\sigma_s)^{\pm 1}$ by $(\sigma_s^{(n)})^{\pm 1}$ and keeping the $S$-letters the same.
If $\mathcal{W}$ is a collection of words in the generating set $P$, then for $n\in \N\cup \{\infty\}$ we define: $$\xi_n(\mathcal{W})=\{\xi_n(W)\mid W\in \mathcal{W}\}$$
Let $\Pi$ be the normal closure of $\{\sigma_s\mid s\in S\}$ in $F(P)$.
Let $(X_n)_{n\in \N}$ be an increasing collection of finite subsets of elements of the group $\Pi$ which exhaust the group.

Now we are ready to build our group.
Let $\sigma_s^{(0)}$ be the involution $(id,s)$ and $k_0=\nu_0=0$. 
Assume that $\{\sigma_s^{(n)}\}_{s\in S},k_n$ and $\nu_n$ have been defined such that:
\begin{enumerate}
\item For each $p<q\leq n$, we have a natural marked isomorphism: 
$$B_{k_{p}}(\Gamma(H_{q},S_q))\cong B_{k_{p}}(\Gamma(H_p,S_p))$$
\item For each $1\leq p\leq n$, $H_p$ is a subgroup of $\langle G,\Sym_f(G)\rangle\cong \Sym_f(G)\rtimes G$. 
\item For each $1\leq p\leq n$, the group $\Pi_p$ is a subgroup of $\Sym_f(G)$, hence it is locally finite. 
Moreover, $k_p$ is also chosen so that $B_{k_{p}}(\Gamma(H_p,S_p))$ contains the relations witnessing that the finitely generated group $\langle \xi_n(X_n)\rangle$ is finite.
\item For each $1\leq p\leq n$, $f\in B_{k_{p-1}}^v(\Gamma(H_{p},S_{p}))$, we have that $f\cdot z^{\nu_{p}}=z^{\nu_{p}}(f\cdot id)$.
\item $k_1<...<k_n$, $\nu_1<...<\nu_n$ and each $\nu_i$ is an even number.
\end{enumerate}
Note that for the base case when $n=0$, this holds trivially.
We will demonstrate the inductive step.

\begin{lemma}\label{Lem:Construction2}
We can choose $\nu_{n+1}\in 2\N$ satisfying that $\nu_{n+1}>\nu_n$, and define:
$$\sigma_s^{(n+1)}=z^{\nu_{n+1}}\sigma_s^{(n)}z^{-\nu_{n+1}}\sigma_s^{(n)}\qquad S_{n+1}=S\cup \{\sigma_s^{(n+1)}\mid s\in S\}\qquad H_{n+1}=\langle S_{n+1}\rangle $$
so that: 
\begin{enumerate}
\item For each $l\leq n$ there is a natural marked isomorphism: $$B_{k_l}(\Gamma(H_{n+1},S_{n+1}))\cong B_{k_l}(\Gamma(H_l,S_l))$$
\item $\sigma_s^{(n+1)}$ is an involution.
\item For $f\in B_{k_{n}}^v(\Gamma(H_{n+1},S_{n+1}))$, we have that $f\cdot z^{\nu_{n+1}}=z^{\nu_{n+1}}(f\cdot id)$.

\end{enumerate}
Furthermore, we can choose $k_{n+1}\in \N$ such that $k_{n+1}>k_n$ and:
\begin{enumerate}
\item[(4)] $B_{k_{n+1}}(\Gamma(H_{n+1},S_{n+1}))$ contains the relations witnessing that the finitely generated group $\langle \xi_{n+1}(X_{n+1})\rangle$ is finite.
\end{enumerate}
\end{lemma}
\begin{proof}
The proof that we can choose a suitable $\nu_{n+1}$ follows from an application of Lemma \ref{lemma:largecommutator}.
The resulting group $H_{n+1}$ is a subgroup of $\langle G,\Sym_f(G)\rangle\cong \Sym_f(G)\rtimes G$ and the group $\Pi_{n+1}$ is a subgroup of $\Sym_f(G)$.
Hence $\Pi_{n+1}$ is locally finite.
Therefore, we can also choose $k_{n+1}>k_n$ such that part $(4)$ holds.
\end{proof}

So by induction, our construction of all three sequences is complete.
By construction, it follows that for each $s\in S$, the limit $\lim_{n\to \infty} \sigma_s^{(n)}$ is defined in $\Sym(G)$. This is an involution that we denote by $\sigma_s^{(\infty)}$.
Recall that we denote: $$S_{\infty}=S\cup \{\sigma_s^{(\infty)}\mid s\in S\}\qquad H_{\infty}=\langle S_{\infty}\rangle\qquad \Pi_{\infty}=\langle \langle \{\sigma_s^{(\infty)}\mid s\in S\}\rangle \rangle_{H_{\infty}}$$
The following is an immediate consequence of the construction and Lemma \ref{Lem:Construction2}.

\begin{prop}\label{Prop:Construction}
For each $n\in \N$, we have a natural marked isomorphism: 
$$B_{k_n}(\Gamma(H_{\infty},S_{\infty}))\cong B_{k_n}(\Gamma(H_n,S_n))$$
In particular, $(H_{\infty},S_{\infty})$ is the limit of $(H_n,S_n)_{n\in \N}$ in the space of marked groups.
\end{prop}


\begin{prop}\label{noF2}
There is a short exact sequence: $$1\to \Pi_{\infty}\to H_{\infty}\to G\to 1$$
such that $\Pi_{\infty}$ is locally finite.
\end{prop}

\begin{proof}
It suffices to show that for each $n\in \N$, the group $\langle \xi_{\infty}(X_n)\rangle$ is finite.
Applying Proposition \ref{Prop:Construction}, we have a natural marked isomorphism:
$$B_{k_n}(\Gamma(H_n,S_n))\cong B_{k_n}(\Gamma(H_{\infty},S_{\infty}))$$ 
We know from Lemma \ref{Lem:Construction2} that the group $\langle \xi_n(X_n)\rangle$ is finite, and the relations witnessing this lie in $B_{k_n}(\Gamma(H_n,S_n))$.
So we have the analogous relations in $B_{k_{n}}(\Gamma(H_{\infty},S_{\infty}))$ and hence conclude that $\langle \xi_{\infty}(X_n)\rangle$ is finite.
\end{proof}

\begin{prop}\label{Prop:HT}
The group $H_{\infty}$ acts highly transitively on $G$.
\end{prop}

In order to prove this, we will need the following Lemmas as ingredients.
\begin{lemma}\label{specialelements1}
For each $s\in S$, there is a sequence of elements $(\sigma_{s,n})_{n\in \N}$ in $H_{\infty}$ satisfying that $\sigma_{s,n}$ is a product of 
the involution $(id,s)$ and some subset of the set of involutions $\{(z^m,z^ms)\mid m\geq n\}$.
\end{lemma}

\begin{proof}
Recall that $S=S_0\cup \{z\}$.
First we handle the case when $s\in S_0$.
Recall that for each $s\in S_0$, $\sigma_s^{(\infty)}$ is a product of a certain collection of disjoint involutions of the form $(z^n,z^ns)$.
We shall construct the sequence inductively, with the base case as $\sigma_{s,0}=\sigma_s^{(\infty)}$.
Assume that we have constructed $\sigma_{s,n}\in H_{\infty}$.
If $\sigma_{s,n}$ fixes $z^{n+1},z^{n+1}s$ then we simply declare $\sigma_{s,n+1}=\sigma_{s,n}$.
If not, then $\sigma_{s,n}$ switches $z^{n+1},z^{n+1}s$.
We let $\sigma_{s,n+1}=z^{(n+1)}\sigma_{s,n}z^{-(n+1)}\sigma_{s,n}$.
It is easy to check that this satisfies the required condition.

In the case when $s=z$, note that $\sigma_z^{(\infty)}$ is a product of a certain collection of disjoint involutions of the form 
$(z^n,z^{n+1})$ where $n$ is an even number. We shall construct the sequence inductively, with the base case as $\sigma_{z,0}=\sigma_z^{(\infty)}$.
Assume that we have constructed $\sigma_{z,n}\in H_{\infty}$ with the additional property that it is a product of a certain collection of disjoint involutions of the form 
$(z^n,z^{n+1})$ where $n$ is an even number. If $\sigma_{z,n}$ fixes $z^{n+1},z^{n+2}$ then we simply declare $\sigma_{z,n+1}=\sigma_{z,n}$.
If not, then $\sigma_{z,n}$ switches $z^{n+1},z^{n+2}$.
We let $\sigma_{z,n+1}=z^{(n+1)}\sigma_{z,n}z^{-(n+1)}\sigma_{z,n}$.
It is easy to check that this satisfies the required condition.
\end{proof}

\begin{lemma}\label{specialelements2}
For each $n\in \N,s\in S$, and each $f\in B_n^v(\Gamma(G,S))$ so that $fs\in B_n^v(\Gamma(G,S))$, there is an element $\sigma_{s,f,n}\in H_{\infty}$ such that 
the restriction: $$\sigma_{s,f,n}\restriction B_n^v(\Gamma(G,S))$$ is the involution switching the elements $\{f,fs\}$ and pointwise fixing: $$B_n^v(\Gamma(G,S))\setminus \{f,fs\}$$
\end{lemma}

\begin{proof}
Let $(\sigma_{s,n})_{n\in \N}$ be the sequence of elements constructed in Lemma \ref{specialelements1}.
For a sufficiently large $m\in \N$, the element $\sigma_{s,f,n}=f\sigma_{s,m}f^{-1}$ satisfies the required conditions.
\end{proof}

\begin{proof}[Proof of Proposition \ref{Prop:HT}]
For $f\in G, s\in S$, let $\beta_{s,f}$ be the involution $(f,fs)$. 
(We remark that $\beta_{s,f}\notin H_{\infty}$, which will be proved subsequently in Proposition \ref{MIF}).
For $n\in \N$, the set of involutions: $$\{\beta_{s,f}\mid f,fs\in B_n^v(\Gamma(G,S)), s\in S\}$$ generate the group of permutations supported on $B_n(\Gamma(G,S))$.
Let $\sigma_{s,f,n}\in H_{\infty}$ be the elements emerging from Lemma \ref{specialelements2}.
Indeed, the group generated by: $$\{\sigma_{s,f,n}\mid f,fs\in B_n^v(\Gamma(G,S))\}$$ leaves $B_n^v(\Gamma(G,S))$ invariant. Moreover: 
$$\sigma_{s,f,n}\restriction B_n^v(\Gamma(G,S))=\beta_{s,f}\restriction B_n^v(\Gamma(G,S))$$ So the restriction of the group generated by: $$\{\sigma_{s,f,n}\mid f,fs\in B_n^v(\Gamma(G,S))\}$$ to $B_n^v(\Gamma(G,S))$ is precisely the action of the group of permutations supported on $B_n^v(\Gamma(G,S))$ and hence acts maximally transitively on the finite set $B_n^v(\Gamma(G,S))$. Therefore, our group $H_{\infty}$ satisfies the desired transitivity properties.
\end{proof}

Our final goal is to show the following.

\begin{prop}\label{MIF}
The group $H$ is mixed identity free.
\end{prop}

Given a word $W$ in the generating set $S_{\infty}=S\cup \{\sigma_s^{(\infty)}\mid s\in S\}$, we denote by $\eta_n(W)$ as the word obtained by replacing each $(\sigma_s^{(\infty)})^{\pm 1}$ by $(\sigma_s^{(n)})^{\pm 1}$ and keeping the generators of $S$ the same.
Note that this is now a word in the generating set $S_n$.
The following Lemma is a direct consequence of our construction.
\begin{lemma}\label{Lem:ActionApprox}
Given an $S_{\infty}$-word $W$, there is an $n\in \N$ such that for each $m\geq n$, we have that:
$$\eta_m(W)\cdot id=W\cdot id\qquad \eta_m(W)\cdot z^{\nu_m}=W\cdot z^{\nu_m}$$
\end{lemma}

\begin{proof}[Proof of Proposition \ref{MIF}]
Recall that Hull and Osin showed that if a countable group has a highly transitive faithful action on $\N$, then either the group contains the alternating group of finitely supported permutations on $\N$ as a normal subgroup, or it is mixed identity free.
(Theorem $5.9$ in \cite{HullOsin}).
So it suffices to show that $H_{\infty}\cap \textup{Sym}_f(G)=\{id\}$.
Assume otherwise. 
Then there is a finitely supported permutation $f\in H_{\infty}$ such that $f\cdot id\neq id$ and $\Supp(f)\subset B_l^v(\Gamma(G,S))$ for some $l\in \N$.   
Let $W$ be an $S_{\infty}$-word that represents $f$.
By Lemma \ref{Lem:ActionApprox} there is an $n\in \N$ such that for each $m\geq n$, we have that:
$$\eta_m(W)\cdot id=W\cdot id\qquad \eta_m(W)\cdot z^{\nu_m}=W\cdot z^{\nu_m}$$
Then by Lemma \ref{Lem:Construction2}: $$W\cdot z^{\nu_m}=\eta_m(W)\cdot z^{\nu_m}=z^{\nu_m}(\eta_m(W)\cdot id)\neq z^{\nu_m}$$
Since $z$ has infinite order and $(\nu_n)_{n\in \N}$ is an increasing sequence, this means that $\Supp(f)\nsubseteq B_l^v(\Gamma(G,S))$ for some $l\in \N$, a contradiction.
\end{proof}

\section*{Acknowledgments}
The second author is supported by NSF Grant DMS-2240136.
The authors thank Nicolas Monod and Denis Osin for valuable feedback. 
The authors also thank Srivatsav Kunnawalkam Elayavalli and Francesco Fournier-Facio for stimulating conversations, and for informing us that it is not known whether there is a 
nonamenable group without free subgroups that is highly transitive and mixed identity free.

\bibliographystyle{amsalpha}
\bibliography{bib}

\providecommand{\bysame}{\leavevmode\hbox to3em{\hrulefill}\thinspace}
\providecommand{\MR}{\relax\ifhmode\unskip\space\fi MR }
\providecommand{\MRhref}[2]{%
  \href{http://www.ams.org/mathscinet-getitem?mr=#1}{#2}
}
\providecommand{\href}[2]{#2}
\begin{thebibliography}{BFFHZ25}

\bibitem[AGEP25]{amrutam2025strictcomparisonreducedgroup}
Tattwamasi Amrutam, David Gao, Srivatsav~Kunnawalkam Elayavalli, and Gregory Patchell, \emph{Strict comparison in reduced group ${C}^*$-algebras}, 2025.

\bibitem[Arz14]{MR3644759}
Goulnara Arzhantseva, \emph{Asymptotic approximations of finitely generated groups}, Extended abstracts {F}all 2012---automorphisms of free groups, Trends Math. Res. Perspect. CRM Barc., vol.~1, Springer, Cham, 2014, pp.~7--15. \MR{3644759}

\bibitem[{Arz}25]{Arzhantseva_personal}
{Arzhantseva, G.}, \emph{Open questions}, personal communication, 2025, September 8, 2025.

\bibitem[BFFHZ25]{belk2025boonehigmanembeddingsmathrmautfnmapping}
James Belk, Francesco Fournier-Facio, James Hyde, and Matthew C.~B. Zaremsky, \emph{Boone-higman embeddings of $\mathrm{Aut}(f_n)$ and mapping class groups of punctured surfaces}, 2025.

\bibitem[BST23]{bradford2023lengthnonsolutionsequationsconstants}
Henry Bradford, Jakob Schneider, and Andreas Thom, \emph{On the length of non-solutions to equations with constants in some linear groups}, 2023.

\bibitem[DM96]{MR1409812}
John~D. Dixon and Brian Mortimer, \emph{Permutation groups}, Graduate Texts in Mathematics, vol. 163, Springer-Verlag, New York, 1996. \MR{1409812}

\bibitem[ES25]{elayavalli2025negativeresolutioncalgebraictarski}
Srivatsav~Kunnawalkam Elayavalli and Christopher Schafhauser, \emph{Negative resolution to the ${C}^*$-algebraic tarski problem}, 2025.

\bibitem[GKEL23]{MR4630775}
Isaac Goldbring, Srivatsav Kunnawalkam~Elayavalli, and Yash Lodha, \emph{Generic algebraic properties in spaces of enumerated groups}, Trans. Amer. Math. Soc. \textbf{376} (2023), no.~9, 6245--6282. \MR{4630775}

\bibitem[HO16]{HullOsin}
Michael Hull and Denis Osin, \emph{Transitivity degrees of countable groups and acylindrical hyperbolicity}, Israel J. Math. \textbf{216} (2016), no.~1, 307--353. \MR{3556970}

\bibitem[LM16]{LodhaMoore}
Y.~Lodha and J.~T. Moore, \emph{A nonamenable finitely presented group of piecewise projective homeomorphisms}, Groups Geom. Dyn. \textbf{10} (2016), no.~1, 177--200. \MR{3460335}

\bibitem[Mon13]{MonodPNAS}
N.~Monod, \emph{Groups of piecewise projective homeomorphisms}, Proc. Natl. Acad. Sci. USA \textbf{110} (2013), no.~12, 4524--4527. \MR{3047655}

\bibitem[MOW21]{MR4286046}
A.~Minasyan, D.~Osin, and S.~Witzel, \emph{Quasi-isometric diversity of marked groups}, J. Topol. \textbf{14} (2021), no.~2, 488--503. \MR{4286046}

\bibitem[Ol'80]{VNDAYFIRST}
A.~Yu. Ol'shanskii, \emph{On the question of the existence of an invariant mean on a group}, Uspekhi Mat. Nauk \textbf{35} (1980), no.~4(214), 199--200. \MR{586204}

\bibitem[OS02]{OlshanskiiSapir}
A.~Yu. Ol'shanskii and Mark~V. Sapir, \emph{Non-amenable finitely presented torsion-by-cyclic groups}, Publ. Math. Inst. Hautes \'{E}tudes Sci. (2002), no.~96, 43--169 (2003). \MR{1985031}

\bibitem[Oza25]{ozawa2025proximalityselflessnessgroupcalgebras}
Narutaka Ozawa, \emph{Proximality and selflessness for group $c^*$-algebras}, 2025.

\end{thebibliography}

\normalsize

\vspace{0.5cm}

James Hyde.

\noindent{\textsc{Department of Mathematics,
Binghamton University.}}

\noindent{\textit{E-mail address:} \texttt{jameshydemaths@gmail.com}}

Yash Lodha.

\noindent{\textsc{Department of Mathematics, Purdue University}}

\noindent{\textit{E-mail address:} \texttt{ylodha@purdue.edu}} 



\end{document}